\documentclass{amsart}

\usepackage{amsfonts,amsthm,amsmath,amssymb,latexsym, enumerate}

\newtheorem{thm}{Theorem}[section]

\newtheorem{lmm}[thm]{Lemma}
\newtheorem{crl}[thm]{Corollary}
\theoremstyle{definition}\newtheorem{rmk}[thm]{Remark}

\theoremstyle{definition}\newtheorem{exa}[thm]{Example}

\theoremstyle{definition}

\def\Br{{\mathrm{Br}}}
\def\End{{\mathrm{End}}}
\def\Hom{{\mathrm{Hom}}}
\def\Ind{\mathrm{Ind}}
\def\Inf{\mathrm{Inf}}
\def\Res{\mathrm{Res}}
\def\Ff{{\mathcal F}}
\def\CF{{\mathcal F}}
\def\Sc{\mathrm{Sc}}
\def\tenk{\otimes_{k}}

\begin{document}

\title[On the Brauer indecomposability of Scott modules]
{On the Brauer indecomposability 
\\
of Scott modules }
\author[R. Kessar, S. Koshitani, M. Linckelmann]
{Radha Kessar, Shigeo Koshitani, Markus Linckelmann} 
\address{City University London, Department of Mathematics,
         Northampton Square, London, EC1V 0HB, United Kingdom
        }
\email{radha.kessar.1@city.ac.uk}        
\address{ Department of Mathematics, Graduate School of Science,
          Chiba University, 1-33 Yayoi-cho, Inage-ku, Chiba,
          263-8522, Japan
        }
\email{koshitan@math.s.chiba-u.ac.jp}
\address{City University London, Department of Mathematics,
         Northampton Square, London, EC1V 0HB, United Kingdom
        }
\email{markus.linckelmann.1@city.ac.uk}
%\maketitle

\begin{abstract}
Let $k$ be an algebraically closed field of prime characteristic $p$,
and let $P$ be a $p$-subgroup of a finite group $G$. We give 
sufficient conditions for the $kG$-Scott module $\Sc(G,P)$ 
with vertex $P$ to remain indcomposable under the Brauer construction 
with respect to any subgroup of $P$.
This generalizes similar results for the case where $P$ is abelian.
The background motivation for this note is the fact that
the Brauer indecomposability of a $p$-permutation bimodule is a key
step towards showing that the module under consideration induces
a stable equivalence of Morita type, which then may possibly
be lifted to a derived equivalence. 
\end{abstract}

\maketitle

%%%%%%%%%%%%%%%%%%%%%%%%%%%%%%%%%%%%%%%%%%%%%%%%%%%%%%%%%%%%%%%%%%%%%
\section{Introduction} \label{intro}

Throughout this paper we denote by 
$k$ an algebraically closed field of prime characteristic $p$.
The Brauer construction with respect to a $p$-subgroup $P$ of a finite 
group $G$ sends a $p$-permutation $kG$-module 
$M$ functorially to a $p$-permutation $kN_G(P)$-module $M(P)$; 
see e.g. \cite[p.402]{Br2} or \cite[pp.91 and 219]{The}.
Following the terminology introduced in \cite{KeKuMi}, the module $M$ is 
called {\it Brauer indecomposable} if the $kC_G(Q)$-module 
$\Res^{N_G(Q)}_{C_G(Q)}(M(Q))$ is indecomposable or zero for any 
$p$-subgroup $Q$ of $G$. As mentioned in \cite{KeKuMi}, the Brauer 
indecomposability of $p$-permutation modules is relevant for the gluing 
technique used for proving categorical equivalences between $p$-blocks 
of finite groups as in Brou\'e's abelian defect group conjecture, 
see \cite{KoLi}, \cite{KeKuMi} and \cite{Tu}. For any subgroup $H$ of 
$G$ there is up to isomorphism a unique indecomposable direct summand 
of the permutation $kG$-module $kG/H$ which has a Sylow $p$-subgroup 
of $H$ as a vertex and the trivial $kG$-module as a quotient. 
This is called the {\it Scott $kG$-module with respect to $H$},
denoted by $\Sc(G,H)$. If $P$ is a Sylow $p$-subgroup
of $H$, then $\Sc(G,H)=$ $\Sc(G,P)$ is, up to isomorphism, the unique 
indecomposable $kG$-module with $P$ as a vertex, the trivial 
$kP$-module as a source, and the trivial $kG$-module as a quotient.
See \cite[Chap.4, \S 8]{NaTsu} and \cite{Br2} for more details
on Scott modules. We first prove a criterion for the Brauer
indecomposability of Scott modules in terms of the indecomposability
of Scott modules of certain local subgroups.

\begin{thm} \label{ScottBrauerindec}
Let $P$ be a $p$-subgroup of a finite group $G$. 
Let $\Ff = \Ff_P(G)$ be the fusion system of $G$ on $P$.
Suppose that $\Ff$ is saturated, and that $\Ff=$ $\Ff_P(N_G(P))$. 
Then $\Sc(G,P)$ is Brauer indecomposable if and only if 
$$\Res^{N_P(Q)C_G(Q)}_{C_G(Q)}(\Sc(N_P(Q)C_G(Q), N_P(Q))$$
is indecomposable for any subgroup $Q$ of $P$.
\end{thm}

It is shown in \cite[Theorem 1.2]{KeKuMi} that if $P$ is an abelian 
$p$-subgroup of $G$, and if the fusion system $\Ff_P(G)$ is 
saturated, then the $kG$-Scott module $\Sc(G,P)$  
is Brauer indecomposable. The following result extends this
in some cases to non-abelian $P$.

\begin{thm} \label{saturationNew} 
Let $P$ be a $p$-subgroup of a finite group $G$. 
Let $\Ff = \Ff_P(G)$ be the fusion system of $G$ on $P$. Suppose that 
$\Ff$ is saturated, and that $\Ff= \Ff_P(N_G(P))$. Suppose that, for every 
subgroup $Q$ of $P$, at least one of the following  holds:  
\begin{enumerate}
\item[{\rm (a)}] $N_P(Q) = QC_P(Q)$.
\item[{\rm (b)}] $C_G(Q)$ is $p$-nilpotent. 
\end{enumerate}
Then  $\Sc(G,P)$ is Brauer indecomposable.
\end{thm}

If $P$ is a common subgroup of  two groups $G$ and $H$, we denote
by $\Delta P$ the `diagonal' subgroup $\Delta P=$ 
$\{(u,u)\ |\ u\in P\}$ of $G\times H$.

\begin{crl} \label{saturation}  
Let $G$ be a finite group and $P$ a Sylow  $p$-subgroup of $G$. 
Set $M= {\mathrm{Sc}}(G\times N_G(P), \Delta P)$.
Suppose that, for every subgroup $Q$ of $P$, at least 
one of the following  holds:  
\begin{enumerate}
\item[{\rm (a)}] $N_P(Q) = QC_P(Q)$.
\item[{\rm (b)}] $C_G(Q)$ is $p$-nilpotent. 
\end{enumerate}
Then  $M$ is Brauer indecomposable.
\end{crl}

\begin{rmk}
For $P$ abelian, this is Corollary 1.4 of \cite{KeKuMi}, which follows 
also from \cite[Theorem]{KoLi}.
Examples of non-abelian $p$-groups to which the above applies 
are all groups of order $p^3$ and metacyclic $p$-groups  of the form 
$M_{n+1}(p) \cong C_{p^n}\rtimes C_p$, see \cite[p.190]{Go},
where $C_m$ denotes a cyclic group of order $m$, for any positive
integer $m$. See the Example \ref{examplep3} below for a stable 
equivalence of Morita type which is constructed making use of Corollary
\ref{saturation}.
\end{rmk} 

The above results will be proved in section 3. We adopt the following 
notation and conventions. All modules over finite group algebras
are assumed to be finitely generated unitary left modules.
We write $H\leq G$ if $H$ is a subgroup of a group $G$, and 
$H\unlhd G$ if $H$ is normal in $G$. The trivial $kG$-module will
be denoted again by $k$. For $G$ a group, $H$ a subgroup of 
$G$, $M$ a $kG$-module and $N$ a $kH$-module, we write as usual 
$\Res^G_H(M)$ for the restriction of $M$ from $kG$ to $kH$ 
and $\Ind_H^G(N)$ for the induction of $N$ from $kH$ to $kG$.
For a subset $S$ of $G$ and an element $g\in G$, we write $^g\!S$ 
for $gSg^{-1}$, and for $h\in G$, we write $^g{\!}h = ghg^{-1}$.
For $H, K \leq G$ we write $H\leq_G K$ when $^g{\!}H \leq K$ for 
an element $g\in G$. As mentioned before, given a $p$-subgroup
$P$ of a finite group $G$ and a $kG$-module $M$, 
we write $M(P)$ for the Brauer construction with respect to $P$
applied to $M$; see \cite[p.402]{Br1} or \cite[pp.91 and 219]{The}. 
We denote by $\Ff_P(G)$ the fusion system of $G$ on $P$; that is,
$\Ff_P(G)$ is the category whose objects are the subgroups of $P$ 
and whose morphisms from $Q$ to $R$ are the group homomorphisms
induced by conjugation by elements of $G$;
see \cite[Definition I.2.1]{AKO} and \cite[p.83]{Li}.
If $P$ is a Sylow $p$-subgroup of $G$, then $\Ff_P(G)$ is {\it saturated},
see \cite[Definition I.2.2]{AKO}. If $\Ff_P(G)=$ $\Ff_P(N_G(P))$,
then the saturation of $\Ff_P(G)$ is equivalent to requiring
that $N_G(P)/PC_G(P)$ has order prime to $p$. For any remaining 
notation and terminology, see the books of \cite{NaTsu} and \cite{The}, 
and also \cite{AKO} and \cite{Li} for fusion systems.

%%%%%%%%%%%%%%%%%%%%%%%%%%%%%%%%%%%%%%%%%%%%%%%%%%%%%%%%%%%%%%%%%%%%%%%%%
\section{Lemmas} \label{lemmas}

This section contains some technicalities needed for the proofs of 
the main results in the section. We start with a very brief review 
of some basic properties of Scott modules. Let $G$ be a  finite group, 
$H$ a subgroup of $G$, and $P$ a  Sylow $p$-subgroup of $H$.  Let $M$ 
be a $p$-permutation $kG$-module. In particular, $M$ has a $k$-basis
$X$ which is permuted by the action of $P$. By \cite[\S 1]{Br2} or
\cite[Proposition (27.6)]{The}, the image in $M(P)$ of the
subset $X^P$ of $P$-fixed points in $X$ is a $k$-basis of
$M(P)$, and we have a direct sum decomposition of $kN_G(P)$-modules
$\Res^G_{N_G(P)}(M)=$ $M(P)\oplus N$, where $N$ is the span of the 
$P$-orbit sums of $X\setminus X^P$. For any subgroup $Q$ of $P$ we 
have $X^P\subseteq$ $X^Q$. In particular, if $M(P)\neq$ $\{0\}$, then 
$M(Q)\neq$ $\{0\}$ for any subgroup $Q$ of $P$. By \cite[(1.3)]{Br2}, 
if $M$ is an indecomposable $p$-permutation $kG$-module, then 
$M(P)\neq$ $\{0\}$ if and only if $P$ is contained  in a vertex of $M$.  
By \cite[(3.2) Theorem]{Br2}, if $P$ is a vertex of $M$, then $M(P)$ 
is the Green correspondent of $M$. Frobenius' reciprocity implies that 
$\Hom_{kG}(\Ind^G_H(k),k)\cong$ $\Hom_{kH}(k,k)\cong$ $k$. Thus exactly 
one indecomposable direct summand of $\Ind^G_H(k)$ has a quotient
isomorphic to the trivial $kG$-module. This summand is the
{\it Scott module} ${\mathrm{Sc}}(G,H)$. 
Under the above isomorphism the identity 
map on $k$ (viewed as a $kP$-module) corresponds to the unique 
$kG$-homomorphism $\eta : \Ind^G_H(k)\to$ $k$ sending 
each $y\otimes 1_k$ to $1$ for any $y\in$ $G$. Thus 
the Scott module $\Sc(G,H)$ is, up to isomorphism, the unique 
indecomposable direct summand of $\Ind^G_H(k)$ which is not contained 
in $\ker(\eta)$. Applying the Brauer construction to $\eta$ yields a 
non-zero map $\eta(P) : (\Ind^G_H(k))(P)\to$ $k$, because the element 
$1\otimes 1_k$ is a $P$-fixed element of the $P$-stable basis consisting 
of the elements $y\otimes 1_k$, with $y$ running over a set of 
representatives of the cosets $G/H$ in $G$. This shows in particular 
that $\Sc(G,H)$ has $P$ as a vertex and therefore must coincide with 
$\Sc(G,P)$. We will use these facts without further reference.
The following lemma is essentially a special case of a result of 
H.~Kawai \cite[Theorem 1.7]{Ka}. 

\begin{lmm} \label{restrict-scott} 
Let $G$ be a finite group, and let $P$ and $Q$ be $p$-subgroups of 
$G$ such that $Q\leq P$. Suppose that for any $g \in G $ satisfying
$Q \leq \,{^g{\!}P}$ we have $|N_{^g{P}}(Q)| \leq |N_P(Q)|$. Let $M$ be 
an indecomposable $p$-permutation $kG$-module with vertex $P$. Set 
$H=$ $N_G(Q)$. Then $\Res^G_H(M)$ has an indecomposable direct 
summand $X$ satisfying $X(N_P(Q))\neq$ $\{0\}$, and any such summand 
has $N_P(Q)$ as a vertex. In particular, 
$\Sc(H, \, N_P(Q))$ is isomorphic to a direct summand of 
$\Res^G_H(\Sc(G,P))$ and of $(\Sc(G,P))(Q)$.
\end{lmm}

\begin{proof}  
We have $H\cap P=$ $N_P(Q)$, and since $M(P)$ is non-zero, so is 
$M(H\cap P)$. Thus there is an indecomposable direct summand $X$ of 
$\Res^G_{H}(M)$ such that $X(H\cap P)\neq$ $\{0\}$.
Let $R$ be a  vertex of $X$ containing $H\cap P$. 
Since $P$ is a vertex of $M$, it follows that
$M$ is isomorphic to a direct summand of $\Ind^G_P(k)$. 
The Mackey decomposition formula implies that $X$ is isomorphic to
a direct summand of
$$
\Res^{G}_{H}(\Ind^G_P(k))=
\bigoplus_y\ \Ind^{H}_{H\cap \, {^y{\!}{P}}}(k)\ ,
$$
where $y$ runs over a set of representatives of the double cosets
$H\backslash G/P$ in $G$. The indecomposability of $X$ and the
Krull-Schmidt theorem imply that there is $y\in$ $G$ such that
$X$ is isomorphic to a direct summand of 
$\Ind^{H}_{H\cap \,{^y{\!}{P}}}(k)$. 
Then $H\cap \, {^y{\!}{P}}$ contains a vertex
$S$ of $X$. Since the vertices of an indecomposable module are 
conjugate, it follows that there is $h\in$ $H$ such that
$S=$ ${^h{\!}{R}}$. The element $h$ normalises $Q$, and hence
$Q\leq S\leq$ $H\cap{^y{\!}{P}}$. 
This implies $S \leq$ $N_{\, {^y{\!}P}}(Q)$.
The assumptions imply further that $|S|\leq$ $|N_P(Q)|\leq$ $|R|$.
Since $R$ and $S$ are conjugate, they have the same order, whence
$R=$ $N_P(Q)$ is a vertex of $X$. For the second statement,
suppose that $M=$ $\Sc(G,P)$. That is, $M$ is, up to isomorphism,
the unique indecomposable direct summand of $kG/P$ which is not in
the kernel of the $kG$-homomorphism $kG/P\to$ $k$ sending
each coset $yP$ to $1$ where $y\in G$. 
As mentioned at the beginning of this
section, the trivial coset $P$ is a $P$-fixed point of the basis of 
$kG/P$ consisting of the $P$-cosets in $G$, and hence applying the 
Brauer construction to a non-zero $kG$-homomorphism $M\to$ $k$
yields a non-zero map $M(P)\to$ $k$. Then also the map
$M(R)\to$ $k$ induced by a non-zero $kG$-homomorphism $M\to$ $k$
is nonzero. It follows that $\Res^G_H(M)$ has an indecomposable
direct summand $X$ satisfying $X(R)\neq$ $\{0\}$ 
such that there is a non-zero $kH$-homomorphism $X\to k$. By the first
statement, $R$ is a defect group of $X$. Thus $X\cong$ $\Sc(H,R)$.
This shows that $\Sc(H,R)$ is isomorphic to a direct summand of
$\Res^G_{N_G(Q)}(M)$. Since $R$ contains $Q$ and $Q$ is
normal in $H$, it follows that $Q$ acts trivially on
$\Sc(H,Q)$, and thus $\Sc(H,Q)$ is isomorphic to a direct summand
of $M(Q)$.
\end{proof} 

In fusion theoretic terminology, the hypothesis on the maximality
of $|N_P(Q)|$ in the previous lemma is equivalent to requiring that
$Q$ is fully $\Ff$-normalised. If $\Ff=$ $N_\Ff(P)$, then every
subgroup of $P$ is fully $\Ff$-normalised, which explains why this
hypothesis is no longer needed in the second statemement of
the next lemma. The proof of the first statement of the next lemma 
is essentially in \cite[Theorem 1.2]{KeKuMi}.  

\begin{lmm}\label{control-scott} 
Let $G$ be a finite group and 
$P$ a $p$-subgroup of $G$. Set $\Ff=\Ff_P(G)$. Assume that $\Ff$ is 
saturated and that $\Ff=$ $\Ff_P(N_G(P))$.  
\begin{enumerate}
  \item[{\rm{(i)}}]
Suppose that $M$ is an indecomposable $p$-permutation $kG$-module
with vertex $P$. Then for any subgroup $Q$ of $P$ and for any 
indecomposable direct summand $X$ of $\Res^G_{N_G(Q)}(M)$ 
satisfying $X(Q)\neq$ $\{0\}$, there is a vertex $R$ of $X$ such that 
$Q\leq R \leq P$.
  \item[{\rm{(ii)}}]
For any subgroup $Q$ of $P$, the module $\Sc(N_G(Q), N_P(Q))$ is a 
direct summand of $\Res_{N_G(Q)}^{G}(\Sc(G,P))$ and of
$(\Sc(G,P))(Q)$. 
\end{enumerate}
\end{lmm}

\begin{proof} 
(i)  Let $X$ be an indecomposable direct summand of $\Res^G_{N_G(Q)}(M)$
such that $X(Q) \ne \{0\}$. There is a vertex $R$ of $X$ such 
that $Q\leq R$. Then $X(R)\neq$ $\{0\}$, hence $M(R)\neq$ $\{0\}$, 
and so $R$ is contained in a vertex of $M$. Since the vertices of $M$ 
are conjugate to $P$, it follows that there is $g\in$ $G$ such
that $Q\leq$ $R\leq$ ${^g{\!}{P}}$. Then ${^{g^{-1}}{\!}Q} \leq P$, which
implies that the map ${^{g^{-1}}{\!}Q} \to Q$ sending
$u$ to ${^g{\!}u}$ is an isomorphism in the fusion system $\Ff$. The 
assumption $\Ff=$ $\Ff_P(N_G(P))$ implies that there is an element 
$h\in N_G(P)$ such that $c=hg^{-1}\in C_G(Q)$. It follows that
$Q = {^cQ} \leq {^c{\!}R} \leq {^{cg}{\!}P} = {^h{\!}P} = P$.
Clearly ${^c{\!}R}$ is also a vertex of $X$, whence the statement. 

(ii) Let $g\in G$ with $Q\leq {^g{\!}P}$. By the argument in
the proof of (i), there is an element $h\in N_G(P)$ such that 
$c = hg^{-1} \in C_G(Q)$. Then $cg=h$ normalises $P$. Thus
conjugation by $c$ induces an isomorphism 
$N_{{^g{\!}P}}(Q)\cong$ $N_P(Q)$; 
in particular, both groups have the same order.
Therefore Lemma \ref{restrict-scott} implies the assertion.
\end{proof}

%%%%%%%%%%%%%%%%%%%%%%%%%%%%%%%%%%%%%%%%%%%%%%%%%%%%%%%%%%%%%%%%%%%%%%%%
\section{Proof of the main result}

Let $G$ be a finite group and let $M$ be an indecomposable 
$p$-permutation $kG$-module with vertex $P$. If $Q$ is a $p$-subgroup
of $G$ which is not conjugate to a subgroup of $P$, then $M(Q)=$ $\{0\}$.
The property of $M(Q)$ being decomposable is invariant under
conjugation of $Q$ in $G$.
Thus if $M$ is not Brauer indecomposable, then there is a subgroup
$Q$ of $P$ such that $M(Q)$ is decomposable as a $kC_G(Q)$-module.
The key step towards proving the main results is the following lemma.

\begin{lmm} \label{notBrauerindec}
Let $G$ be a finite group and  $P$ a $p$-subgroup of $G$. Set 
$\Ff=\Ff_P(G)$. Assume that $\Ff= \Ff_P(N_G(P))$ and that $\Ff$ is 
saturated. Set $M=$ $\Sc(G,P)$. Suppose that $M$ is not 
Brauer indecomposable. Let $Q$ be a subgroup of maximal order in $P$ 
such that $\Res^{N_G(Q)}_{C_G(Q)}(M(Q))$ is decomposable. Then $Q$ is 
a proper subgroup of $P$ and setting $R=$ $N_P(Q)$, we have 
$$\Res^{N_G(Q)}_{RC_G(Q)}(M(Q)) \cong \Sc(RC_G(Q),R)\ .$$
In particular, $\Res^{N_G(Q)}_{RC_G(Q)}(M(Q))$ is indecomposable
with $R$ as a vertex.
\end{lmm}

\begin{proof}
The $kC_G(P)$-module $M(P)$ is indecomposable by 
\cite[Lemma 4.3(ii)]{KeKuMi} (this lemma requires the hypothesis
on $\Ff$ being saturated). Thus $Q$ is a proper subgroup of $P$, 
hence a proper subgroup of $R=$ $N_P(Q)$, by 
\cite[Chap.1, Theorem 2.11(ii)]{Go}. We first show that $M(Q)$ is 
indecomposable as a $kN_G(Q)$-module. By Lemma \ref{control-scott} (ii) 
we have $\Res^G_{N_G(Q)}(M)=$ $\Sc(N_G(Q),R)\oplus X$ for some
$kN_G(Q)$-module $X$, and $\Sc(N_G(Q),R)$ is isomorphic to a direct
summand of $M(Q)$. We need to show that $X(Q)=$ $\{0\}$.
Arguing by contradiction, suppose that $X(Q)\neq\{0\}$.
Then there exists an indecomposable direct summand $Y$ 
of $X$ such that $Y(Q)\neq \{0\}$. Since $\Ff$ is saturated and 
$N_{\Ff}(P) = \Ff$, it follows from Lemma~\ref{control-scott}(i) 
that $Y$ has a vertex $S$ such that $Q \leq S \leq P$.
Then $S \leq N_G(Q)\cap P = R$. Note that $Q$ is not a vertex
of $M$ since $|Q| \, \ne \, |P|$.
If $Q = S$, then $Q$ is a vertex of $Y$, and hence $Q$ is a vertex
of $M$ by the result of Burry-Carlson-Puig
\cite[Chap. 4, Theorem 4.6(ii)]{NaTsu}, a contradiction.
Thus $Q$ is a proper subgroup of $S$. Since $Y$ is an indecomposable 
$p$-permutation $kN_G(Q)$-module with vertex $S$, we have $Y(S)\neq$
$\{0\}$, and hence $X(S)\neq$ $\{0\}$.
Since $R$ is a vertex of
$\Sc(N_G(Q), R)$ and $S\leq R$, it follows that
$\Big(\Sc(N_G(Q), R)\Big)(S)\neq$ $\{0\}$.
We have 
$$ 
\Big(\Res^G_{N_G(Q)}(M)\Big)(S)
\ = \
\Big( \Sc(N_G(Q), R)\Big)(S) \bigoplus X(S),
$$
and both of the two direct summands of the right hand side are 
non-zero. This implies that  $\Big(\Res^G_{N_G(Q)}(M)\Big)(S)$
is not indecomposable; in other words, $M(S)$ is not indecomposable
as a $k(N_G(Q)\cap N_G(S))$-module.
Since $C_G(S) \le C_G(Q)\cap N_G(S) \le N_G(Q)\cap N_G(S)$,
it follows that $M(S)$ is not indecomposable as a $kC_G(S)$-module.
But this contradicts the assumptions since 
$|P:S| < |P:Q|$. This shows that $X(Q) = \{0\}$, and hence that
$M(Q)$ is indecomposable as a $kN_G(Q)$-module. Using
Lemma \ref{control-scott} (ii), this shows that
$$
   M(Q) = \mathrm{Sc}(N_G(Q), R).
$$  
Set $L=$ $RC_G(Q)$. Since $\Ff=$ $\Ff_P(N_G(P))$, it follows that 
$N_G(Q)=$ $(N_G(Q)\cap N_G(P))C_G(Q)$. The subgroup $N_G(P)\cap N_G(Q)$ 
normalises $R$, and hence $L$ is a normal subgroup of $N_G(Q)$ and we 
have $N_G(Q)=$ $(N_G(R)\cap N_G(Q))L$.  In particular, $L$ acts 
transitively on the set of $N_G(Q)$-conjugates of $R$. 
Since $M(Q)$ has $R$ as a vertex and $R\leq L$, there is an 
indecomposable $kL$-module $V$ with vertex $R$ such that $M(Q)$ is 
isomorphic to a direct summand of $\Ind^{N_G(Q)}_L(V)$. The Mackey 
formula, using that $L$ is normal in $N_G(Q)$, implies that 
$$
\Res^{N_G(Q)}_L(M(Q)) = \bigoplus_x\ {^x{V}}
$$
with $x$ running over a subset $E$ of $N_G(Q)\cap N_G(R)$.
In particular, all indecomposable direct summands of
$\Res^{N_G(Q)}_L(M(Q))$ have $R$ as a vertex. Thus
applying the Brauer construction with respect to $R$ sends
every summand to a non-zero $kN_L(R)$-module.
Therefore, if the set $E$ has more than one element, then 
$M(Q)(R)=M(R)$ is decomposable as a $kN_L(R)$-module,
hence also as a $kC_G(R)$-module. This contradicts the assumptions,
and hence $X$ consists of a single element, or equivalently,
$\Res^{N_G(Q)}_L(M(Q))$ is indecomposable. Then necessarily
$\Res^{N_G(Q)}_L(M(Q))\cong$ $\Sc(L,R)$, whence the result.
\end{proof}

\begin{proof}[{Proof of Theorem \ref{ScottBrauerindec}}]
Set $M=$ $\Sc(G,P)$.
Suppose that $M$ is Brauer indecomposable. Then $M(Q)=$
$\Sc(N_G(Q), N_P(Q))$ by Lemma \ref{control-scott} (ii), and
$M(Q)$ is indecomposable as a module for any subgroup of $N_G(Q)$ 
containing $C_G(Q)$. In particular, setting $M_Q =$ 
$\Sc(N_P(Q)C_G(Q), N_P(Q))$, we have $M_Q\cong$  
$\Res^{N_G(Q)}_{N_P(Q)C_G(Q)}(M(Q))$. By the assumptions, 
the restriction to $kC_G(Q)$ of this module remains
indecomposable. Suppose conversely that 
$\Res^{N_P(Q)C_G(Q)}_{C_G(Q)}(M_Q)$ remains indecomposable for
all subgroups $Q$ of $P$. Arguing by contradiction, suppose that
$M$ is not Brauer indecomposable.
Let $Q$ be a subgroup of maximal order of $P$ such that
$\Res^{N_G(Q)}_{C_G(Q)}(M(Q))$ is decomposable. Set $R=$ $N_P(Q)$.
By Lemma \ref{notBrauerindec}, the $kRC_G(Q)$-module
$\Res^{N_G(Q)}_{RC_G(Q)}(M(Q))$ is indecomposable with vertex $R$, 
hence isomorphic to $M_Q=\Sc(RC_G(Q), R)$ by Lemma 
\ref{control-scott} (ii). Thus $\Res^{RC_G(Q)}_{C_G(Q)}(M_Q)$ is 
decomposable, contradicting the assumptions.
\end{proof}

\begin{proof}[{Proof of Theorem \ref{saturationNew}}]
Set $M=$ $\Sc(G,P)$.
Arguing by contradiction, let $Q$ be a subgroup of maximal order in 
$P$ such that $M(Q)$ is not indecomposable as a $kC_G(Q)$-module.
Set $R=$ $N_P(Q)$ and $L=$ $RC_G(Q)$.
It follows from Lemma \ref{notBrauerindec} that $Q$ is a proper
subgroup of $P$, and that $\Res^{N_G(Q)}_L(M(Q))$ is indecomposable,
with $R$ as a vertex, hence isomorphic to $\Sc(L,R)$ by Lemma
\ref{control-scott} (ii). By the construction of $M(Q)$, the group $Q$ 
acts trivially on $M(Q)$.

\medskip
Suppose first that hypothesis {\bf (a)} holds; that is, $R=$ 
$QC_P(Q)$. Then $L=$ $QC_G(Q)$. Thus 
$\Res^{N_G(Q)}_{QC_G(Q)}(M(Q))$ is indecomposable. Since
$Q$ acts trivially, it follows that
$\Res^{N_G(Q)}_{C_G(Q)}(M(Q))$ is indecomposable, a contradiction.

\medskip
Thus hypothesis {\bf (b)} holds; that is, $C_G(Q)$ is $p$-nilpotent.
The indecomposable $kN_G(Q)$-module $M(Q) =$ $\Sc(N_G(Q),R)$ is in the 
principal block as a $kN_G(Q)$-module, and its restriction to $L=$
$RC_G(Q)$ remains indecomposable by the above. Hence we can assume that 
$O_{p'}(N_G(Q)) = 1$. Then also $O_{p'}(C_G(Q)) = 1$. This implies that 
$C_G(Q)$ is a $p$-group by (b). Hence the groups
$C_G(Q)$, $C_G(R)$, $L=RC_G(Q)$, and $QC_G(Q)$
are all finite $p$-groups.  Using that transitive permutation
modules of finite $p$-groups are indecomposable, it follows
that
$$ \Sc(L,R) = \Res^{N_G(Q)}_{L}(M(Q)) \cong \Ind_R^{L}(k)\ . $$
The Mackey formula implies that
\begin{align*}
\Res^{N_G(Q)}_{QC_G(Q)}(M(Q)) &=
\Res^{L}_{QC_G(Q)}
\, \circ \, \Res^{N_G(Q)}_{L}(M(Q))
\\
&= \Res^{L}_{QC_G(Q)}
\, \circ \, \Ind_R^{L}(k)
\\
&= \Ind_{QC_G(Q)\cap R}^{QC_G(Q)}(k),
\end{align*}
since there is a single double coset here, and so only
one term in the Mackey formula.
This is again a transitive permutation module of the
$p$-group $QC_G(Q)$, hence indecomposable.
As before, since $Q$ acts trivially on $M(Q)$, this implies that
$\Res^{N_G(Q)}_{C_G(Q)}(M(Q))$ is indecomposable.
This concludes the proof.
\end{proof}

\begin{proof}[{Proof of Corollary \ref{saturation}}]
Set $H=$ $N_G(P)$.
The fusion system of $G\times H$ on $\Delta P$ is equal to 
that of $\Delta H$ on $\Delta P$, and this is saturated
as $P$ is a Sylow $p$-subgroup of $H$. Moreover, for
$Q$ a subgroup of $P$, we have $C_{G\times H}(\Delta Q)=$
$C_G(Q)\times C_H(Q)$. Thus if $C_G(Q)$ is $p$-nilpotent, then
so is $C_{G\times H}(\Delta Q)$.
The result follows from Theorem~\ref{saturationNew}.
\end{proof}

\begin{exa} \label{examplep3}
Suppose that $p=3$. Let $G$ be a finite group. Assume that $G$ has a 
Sylow $3$-subgroup $P$ such that $P \cong M_3(3)$, the extraspecial 
$3$-group of order $27$ of exponent $9$. Set $H=$ $N_G(P)$. Then 
the $k(G\times H)$-Scott module $M=$ $\Sc(G\times H, \, \Delta P)$ induces 
a stable equivalence of Morita type between the principal blocks 
$B_0(kG)$ and $B_0(kH)$. This is trivial if $G$ is $3$-nilpotent because
both blocks are isomorphic to $kP$ in that case.
If $G$ is not $3$-nilpotent, then $|N_G(P)/PC_G(P)| = 2$. Let $Q$ be a 
non-trivial subgroup of $P$. It follows from 
Theorem~\ref{saturationNew}, results of Hendren
\cite[Propositions 5.12 and 5.13]{He} and the $Z_3^*$-theorem 
that $M(Q)$ induces a Morita equivalence between $B_0(kC_G(Q))$ and 
$B_0(kC_H(Q))$. Hence the gluing theorem \cite[Theorem 6.3]{Br1} 
implies that $M$ induces a stable equivalence of Morita type between
the principal blocks of $kG$ and $kH$. Furthermore, by 
\cite[Proposition 5.3]{Br1}, such a stable equivalence of Morita
type implies the equality $\mathrm{k}(B_0(kG)) - \ell(B_0(kG)) =$ 
$\mathrm{k}(B_0(kH)) - \ell(B_0(kH))$, where
$\mathrm{k}(B_0(kG))$ and $\ell(B_0(kG))$ denote the number of
ordinary and modular irreducible characters $B_0(kG)$,
respectively, with the analogous notation for $H$ instead of $G$. 
This yields a proof of a special case of a result of Hendren 
\cite[Theorem 5.14]{He}: if $G$ is not $3$-nilpotent, then 
$\mathrm{k}(B_0(kG)) - \ell(B_0(kG)) =$ $8$. 
\end{exa}

\begin{rmk}
%With the notation of Corollary \ref{saturation}, 
Let $G$ be a finite group and $P$ a Sylow $p$-subgroup of $G$.
The Scott module 
$\Sc(G\times N_G(P),\Delta P)$ is the Green correspondent
of the Scott module $\Sc(G\times G,\Delta P)$, which is isomorphic
to the prinicipal block of $kG$ viewed as a $k(G\times G)$-module. 
One might wonder how to generalise Corollary \ref{saturation} 
to arbitrary blocks. 
Let $b$ be a block of $kG$ and let $(P,e_P)$ be
a maximal $(G,b)$-Brauer pair. Set $H=$ $N_G(P,e_P)$. The
$(G\times H)$-Green correspondent with vertex $\Delta P$ of the
$k(G\times G)$-module $kGb$ is of the form $M=$ $kG f$ for
some primitive idempotent $f$ in $(kGb)^{\Delta H}$ satisfying
$\Br_{\Delta P}(f)e_P\neq$ $0$ (see e.g. \cite{ALR}).  
Note that $(P,e_P)$ is also a maximal $(H,e_P)$-Brauer pair.
For any subgroup $Q$ of $P$ denote by $e_Q$ the unique block of
$kC_G(Q)$ satisfying $(Q,e_Q)\leq$ $(P,e_P)$ and by $f_Q$ the
unique block of $kC_H(Q)$ satisfying $(Q,e_Q)\leq$ $(P,e_P)$.
The `obvious' generalisation of Corollary \ref{saturation} would be the
statement that the $kC_G(Q)e_Q$-$kC_H(Q)f_Q$-bimodule 
$e_QM(\Delta Q)f_Q$ is indecomposable. This is, however, not the
case in general. In order to construct an example for which this is
not the case, we first translate this indecomposability to the
source algebra level. 

Let $j\in$ $(kHe_P)^{\Delta P}$ be a source idempotent $e_P$ as a block
of $kHe_P$. Then $i=$ $jf$ is a source idempotent of $kGb$ (see
e.g. \cite[4.10]{FP}). Thus multiplication by $f$ induces an
interior $P$-algebra homomorphism from $B=$ $jkHj$ to $A=$ $ikGi$.
In particular, $A$ can be viewed as an $A$-$B$-bimodule. 
Multiplication by a source idempotent, or more generally, by
an almost source idempotent, is a Morita equivalence (cf.
\cite[3.5]{Puigpointed} and \cite[4.1]{Liperm}). Moreover,
the Brauer construction with respect to a fully $\CF$-centralised
subgroup $Q$ of $P$ sends the source idempotent $i$ to the almost source
idempotent $\Br_{\Delta Q}(i)$ in $kC_G(Q)e_Q$ (cf. \cite[4.5]{Liperm}). 
Through the appropriate Morita equivalences, the $kGb$-$kHe_P$-bimodule
$M=$ $kGf$ corresponds to the $A$-$B$-bimodule $iMj=$ $ikGjf=$ $A$, and 
the $kC_G(Q)e_Q$-$kC_H(Q)f_Q$-bimodule $e_QM(\Delta Q)f_Q$ corresponds
to the $A(\Delta Q)$-$B(\Delta Q)$-bimodule $A(\Delta Q)$.
It follows that for $Q$ a fully $\CF$-centralised subgroup of $P$,
the indecomposability of $e_Q M(\Delta Q) f_Q$ is equivalent to
the indecomposability of $A(\Delta Q)$ as an 
$A(\Delta Q)$-$B(\Delta Q)$-module.

We construct an example for which this fails. Suppose that $p$ is odd.
Let $P$ be an extraspecial $p$-group of order $p^3$ of exponent $p$.
Let $Q$ be a subgroup of order $p^2$ in $P$; we have $C_P(Q)=$ $Q$
and in particular, $Q$ is fully centralised (even centric) with respect 
to any fusion system on $P$. Set $V=$
$\Inf^P_{P/Q}(\Omega_{P/Q}(k))$. Thus $\dim_k(V)=$ $p-1$, and
$Q$ acts trivially on $V$. Setting $S=$ $\End_k(V)$, it follows
that $S=$ $S^{\Delta Q}\cong$ $S(\Delta Q)$. By the main result of
Mazza in \cite{Mazza}, there exists a nilpotent block of some finite
group having a source algebra isomorphic to $A=$ $S\tenk kP$. The
Brauer correspondent of such a block has source algebra $B=$ $kP$.
We have $A(\Delta Q)$ $(S\tenk kP)(\Delta Q)\cong$ $S\tenk kQ$ 
and $B(\Delta Q)=$ $kQ$. Thus any primitive idempotent
$e$ in $S=$ $S^{\Delta Q}$ determines a nontrivial direct bimodule
summand $Se\tenk kQ$ of $A(\Delta Q)$, and hence $A(\Delta Q)$ is not 
indecomposable as an $A(\Delta Q)$-$B(\Delta Q)$-module. 
\end{rmk}

%%%%%%%%%%%%%%%%%%%%%%%%%%%%
\noindent{\bf Acknowledgements.}
{\small
A part of this work was done while the second author was visiting the
first and the third authors in the University of Aberdeen in February 2012
and in the City University London in October 2013. He would like to thank the
Mathematics Institutes for their wonderful hospitality.
The second author was supported by the Japan Society for Promotion of Science 
(JSPS) Grant-in-Aid for Scientific Research(C) 23540007, 2011--2014.}
The authors would like to thank the referee for her/his comments on the
first version.
%%%%%%%%%%%%%%%%%%%%%%%%%%%%%%%
\bibliographystyle{amsalpha}  %

\end{document}